\documentclass[12pt,reqno]{amsart}
\usepackage{amssymb}
\usepackage{amscd}
\usepackage{enumerate}

\def \C{{\mathbb C}}

\def \I{{\mathcal I}}
\def \N{{\mathbb N}}
\def \O{{\mathcal O}}
\def \P{{\mathbb P}}
\def \R{{\mathbb R}}
\def \S{\mathbb{R}/ \mathbb{Z}}
\def \FU{{\frak U}}
\def \FV{{\frak V}}
\def \FW{{\frak W}}
\def \az{\alpha}
\def \bz{\beta}
\def \dz{\delta}
\def \ez{\varepsilon}
\def \fz{\varphi}
\def \Gz{\Gamma}
\def \gz{\gamma}
\def \Lz{\Lambda}
\def \Oz{\Omega}
\def \oz{\omega}
\def \pz{\bar{\partial}}
\def \wz{\infty}
\def \bx{\Box}
\def \st{\subset}

\def \mlist#1.#2.{{#1}_0,\ldots,{#1}_{#2}}
\def \llist#1.#2.{{#1}_1,\ldots,{#1}_{#2}}

\DeclareMathOperator{\Hom}{Hom}
\DeclareMathOperator{\id}{id}
\DeclareMathOperator{\Ker}{Ker}
\DeclareMathOperator{\rank}{rank}
\DeclareMathOperator{\codim}{codim}

\newcommand{\norm}[1]{\left\Vert#1\right\Vert}

\theoremstyle{definition}
\newtheorem*{thmq}{Theorem}
\newtheorem{thm}{Theorem}[section]
\newtheorem{lem}[thm]{Lemma}
\newtheorem{cor}[thm]{Corollary}
\newtheorem{prop}[thm]{Proposition}
\newtheorem{defn}[thm]{Definition}
\numberwithin{equation}{section}

\begin{document}

\title{A splitting theorem for holomorphic Banach bundles}
\author{Jaehong Kim}
\thanks{This research was partially supported by NSF grants DMS0203072 and DMS0700281. I would like to express my best gratitude to Professor L\'aszl\'o Lempert for his suggestions. This paper contains the result of my thesis research under his guidance.
}


\address{Department of Mathematics, Purdue University, West Lafayette, IN 47907}
\email{jhkim@math.purdue.edu}

\begin{abstract}
This paper is motivated by Grothendieck's splitting theorem. In the 1960s, Gohberg generalized this to a class of Banach bundles. We consider a compact complex manifold $X$ and a holomorphic Banach bundle $E \to X$ that is a compact perturbation of a trivial bundle in a sense recently introduced by Lempert. We prove that $E$ splits into the sum of a finite rank bundle and a trivial bundle, provided $H^{1}(X, \O)=0$.
\end{abstract}

\maketitle


\section{Introduction}
\label{intro}
This paper deals with holomorphic vector bundles over complex manifolds. Let us recall two theorems in the subject that are relevant for this paper. The first is (a special case of) the finiteness theorem of Cartan-Serre \cite{CS} : if $X$ is a compact complex manifold with $\dim X = n$ and $V$ is a holomorphic vector bundle of finite rank over $X$, then
\begin{align}
\label{hodge1}
\dim H^{q}(X,V) < \infty \text{ , for  } 0 \le q \le n \text{.}
\end{align}
The other result is Grothendieck's splitting theorem, useful in many areas of geometry and physics.
\begin{thmq}[Grothendieck \cite{Gr}]
Every finite rank holomorphic vector bundle over the Riemann sphere splits into the sum of line bundles.
\end{thmq}

For bundles of infinite rank over finite dimensional, indeed compact manifolds, which are in our focus, certain generalizations of the finiteness and splitting theorems were first proved by Gohberg and Leiterer, see \cite{G}, \cite{GL1}, and \cite{GL2}. In order to explain their results and also our findings, we need to introduce some notions of infinite dimensional analysis. A complex Banach manifold $X$ is a Hausdorff space with an open cover $X=\bigcup U_{i}$, and homeomorphisms $\psi_{i}$ from $U_{i}$ to open sets in a complex Banach space such that all transition mappings $\psi_{i} \psi_{j}^{-1}$ are biholomorphisms where they are defined. We call $(U_{i}, \psi_{i})$ a coordinates system of $X$. Let $X$ be a compact Banach manifold, then $\dim X < \infty$, because balls in infinite dimensional Banach spaces are not precompact. A holomorphic Banach bundle over $X$ is a complex Banach manifold $E$ together with a holomorphic map $\pi : E \to X$ and a vector space structure on each fiber $E_{x} = \pi^{-1}(x)$. It is required that this structure should be locally trivial, i.e. there should exist biholomorphisms $\fz_{i} : E|U_{i} \to U_{i} \times B$, where $U_{i} \subset X$ is an open set and $B$ is a complex Banach space, and $\fz_{i}|E_{x} : E_{x} \to \{x\} \times B$ is a vector space isomorphism, $x \in U_{i}$. As with finite rank bundles, the local trivializations $\fz_{i}$ give rise to functions $\fz_{i}\fz_{j}^{-1} : (U_{i} \cap U_{j}) \times B \to (U_{i} \cap U_{j}) \times B$, which are of the form $\fz_{i}\fz_{j}^{-1}(x,v) = (x, \fz_{ij}(x)v)$. Here $\fz_{ij} : U_{i} \cap U_{j} \to \text{GL}(B)$ is a holomorphic map to the group of invertible linear transformations of $B$. Conversely, given holomorphic transition functions $\fz_{i j} : U_{i} \cap U_{j} \to \text{GL}(B)$ such that $\fz_{i j}\fz_{j k}=\fz_{i k}$, a holomorphic Banach bundle can be defined by gluing. If the transition functions $\fz_{ij}$ are of the form $\fz_{ij}(x) = \id + \text{ compact operator}$, then we say the bundle is of compact type.

\begin{thmq}[Gohberg \cite{G}]
Any holomorphic Banach bundle of compact type over the Riemann sphere splits into a finite sum of line bundles and a trivial Banach bundle.
\end{thmq}
Our main goal is to extend this result to other manifolds. In fact, following Lempert \cite{Le3}, we shall consider bundles slightly more general than compact type. Let $E$, $F$ be Banach bundles over a complex manifold $X$. We say $E$ is a compact perturbation of $F$ if there exist an open cover $\FU=\{U_{i}\}$ of $X$, homomorphisms $\tilde{\fz}_{i}:E|U_{i} \to F|U_{i}$ for every $U_{i} \in \FU$ which are Fredholm operators on each fiber $E_{x}$, $x \in U_{i}$, and $\tilde{\fz}_{i}-\tilde{\fz}_{j}$ are compact operators on each fiber $E_{x}$, $x \in U_{i} \cap U_{j}$. Every finite rank bundle is a compact perturbation of a trivial Banach bundle and if $E$ is of compact type, then $E$ is a compact perturbation of a trivial bundle. \\

Our main result is

\begin{thm}
\label{thm1}
Let $X$ be a compact complex manifold with $H^{1}(X,\O)=0$, and $E \to X$ a holomorphic Banach bundle. If $E$ is a compact perturbation of a trivial bundle then it splits into the sum of a trivial Banach bundle and a finite rank bundle.
\end{thm}
It turns out that $H^{1}(X, \O)=0$ is a necessary condition. Section 4 of \cite{Le3} contains an example, due to V\^aj\^aitu, that shows that if $X$ is K\"ahler and $H^{1}(X,\O) \ne 0$, then there is a compact perturbation $F$ of a trivial Banach bundle $T \to X$ with infinite rank such that $H^{0}(X,F) = 0$. Hence $F$ cannot have any trivial subbundle.


\section{Basic Notions}
\label{sec:1}
Let $X$ be a finite dimensional complex manifold, $\C \otimes TX = T^{1,0}X \oplus T^{0,1}X \to X$ the complexified tangent bundle of $X$, and $E \to X$ a holomorphic Banach bundle. An $E$-valued $r$-form $f$ is a map
\[ f : \bigoplus ^{r}(\C \otimes TX) \to E \text{,}\]
whose restriction to any fiber $\bigoplus ^{r}(\C \otimes T_{x}X)$ is a continuous, alternating, and complex $r$-linear map. It is called a $(0,r)$-form if $f(\llist\xi.r.)=0$ whenever $\xi_{j} \in T^{1,0}_{x}X$ for some $1 \le j \le r$. So $0$-forms are global sections. Let $\{(U_{i}, \psi_{i})\}_{i}$ be an atlas of $X$ such that there are trivializations $\fz_{i} : E|U_{i} \to U_{i} \times B$, where $(B, \norm{\phantom{ab}})$ is a Banach space. Fix smooth vector fields $\llist \xi.m.$ on $X$ which span each tangent space $T_{x}X$. Let $\Oz^{r}(X,E)$ be the space of smooth $E$-valued $(0,r)$-forms with the following seminorms, one for each compact $K \st U_{i}$ and $k=0,1,2, \ldots$,
\begin{align}
\label{norm1}
\norm{f}_{K,k}=\sup\norm{\xi_{j_{1}}\cdots\xi_{j_{s}}(\fz_{i} f(\xi_{j_{s+1}}, \ldots,\xi_{j_{s+r}}))(x)} \text{,}
\end{align}
the sup taken over all tuples $\xi_{j_{1}}, \ldots, \xi_{j_{s+r}}$, $0 \le s \le k$, (if $s=0$, set $\xi_{j_{1}} \cdots \xi_{j_{0}} = 1$), and $x \in K \st U_{i}$. Then $\Oz^{r}(X,E)$ is a Fr\'echet space whose topology is independent of the choices made. If $X$ is an open set in $\C^{n}$ and $T=X \times B$ is a trivial Banach bundle, then $f \in \Oz^{r}(X,T)$ can be expressed as
\[f = \sum_{J} f_{J} d\bar{z}^{J} \text{,} \]
where the $f_{J}$ are $B$-valued functions and $|J|=r$. A $\C$-linear operator $\pz_{T,r} : \Oz^{r}(X,T) \to \Oz^{r+1}(X,T)$ is defined by
\[\pz_{T,r} f = \sum_{i,J} \frac{\partial f_{J}}{\partial \bar{z_{i}}} d\bar{z}_{i} \wedge d\bar{z}^{J} \text{.}\]
For a general holomorphic Banach bundle $E \to X$, $\pz = \pz_{E,r} : \Oz^{r}(X,E) \to \Oz^{r+1}(X,E)$ is defined  through local charts and trivializations as above by
\[\pz f|U_{i} = \fz^{-1}_{i}(\psi^{-1}_{i})^{*} \pz_{T,r} \psi^{*}_{i} \fz_{i} (f|U_{i}) \text{.}\]
Since $\pz$ is a continuous operator, $Z_{\pz}^{r}(X,E) = \Ker \pz_{E,r}$ is a Fr\'echet space. We endow the cohomology groups
\[H^{r}_{\pz}(X,E) = Z_{\pz}^{r}(X,E) / \pz \Oz^{r-1}(X,E) \]
with the induced quotient topology. Thus $H^{r}_{\pz}(X,E)$ is a complete locally convex space but not necessarily Hausdorff.


\section{Dolbeault Isomorphism}
\label{sec:2}
In this section, we will recall how the sheaf cohomology groups $H^{q}(X,E)$ can be endowed with a locally convex topology, and show that Dolbeault's isomorphism $H^{q}_{\pz}(X,E) \cong H^{q}(X,E)$ is an isomorphism of locally convex topological vector spaces.\\

Fix smooth vector fields $\llist \xi.m.$ that span each tangent space $T_{x}X$. Let $\fz_{i} : E|U_{i} \to U_{i} \times B$ be a trivialization over $U_{i} \st X$, where $(B, \norm{\phantom{ab}})$ is a Banach space. Further let $\FU$ be a countable cover of $X$ consisting of such open sets. As in \cite{Le3}, we consider the double complex
\[C^{qr}(\FU) = C^{qr}(\FU,E) = \prod_{\mlist U.q. \in \FU}\Oz^{r}(\bigcap_{i=0}^{q}{U_{i},E})\]
with \v Cech coboundary $\dz = \dz^{qr} : C^{qr}(\FU) \to C^{q+1,r}(\FU)$ and componentwise $\pz = \pz^{qr} : C^{qr}(\FU) \to C^{q,r+1}(\FU)$. Set $Z^{qr}(\FU) = Z^{qr}(\FU, E) = \Ker \pz^{qr} \cap \Ker \dz^{qr}$. We endow the \v Cech cohomology groups
\[H^{q}(\FU,E) = Z^{q0}(\FU,E) / \dz (\Ker \pz \cap C^{q-1,0}(\FU,E)) \]
with the induced quotient topology. Also we endow the sheaf cohomology groups $H^{q}(X,E)$ with the topology induced by the direct limit of the system $\{ H^{q}(\FU,E), \gz^{\FU}_{\FV} \}$, where $\gz^{\FU}_{\FV}$ are the refinement homomorphisms.

\begin{thm}
\label{thm2}
Let $X$ be a finite dimensional complex manifold, and $E \to X$ a holomorphic Banach bundle. Then $H^{q}_{\pz}(X,E) \cong H^{q}(X,E)$ as locally convex topological vector spaces.
\end{thm}

For the proof we have to introduce a subcomplex of $C^{qr}(\FU)$. With $U_{i}$, $\fz_{i}$, $\llist \xi. m.$ as above, and $W \st \st U_{i}$ an open subset, consider the space of bounded $E$-valued $(0,r)$-forms $\Oz^{r}_{b}(W,E) \st \Oz^{r}(W,E)$ with the following norm
\[\norm{f}_{b}=\sup\norm{\fz_{i}f(\xi_{j_{1}}, \ldots, \xi_{j_{r}})(x)} \text{,}\]
the sup taken over all tuples $\xi_{j_{1}}, \ldots, \xi_{j_{r}}$ and $x \in W$. Given a cover $\FW$ consisting of such open sets, we define a double complex
\[C^{qr}_{b}(\FW) = C^{qr}_{b}(\FW,E) = \prod_{\mlist W.q. \in \FW}\Oz^{r}_{b}(\bigcap_{i=0}^{q}{W_{i},E})\]
with $\dz$ and $\pz$ as before. Set $Z^{qr}_{b}(\FW) = Z^{qr}_{b}(\FW,E) = Z^{qr}(\FW,E) \cap C^{qr}_{b}(\FW,E)$ and endow the cohomology groups
\[H^{q}_{b}(\FW,E) = Z^{q0}_{b}(\FW,E) / \dz (\Ker \pz \cap C^{q-1,0}_{b}(\FW,E)) \]
with the induced quotient topology. \\

The following definition comes from \cite{Le3}:
\begin{defn}
We say that a countable open cover $\FW$ of $X$ is special (with respect to $E$) if (a) for every $W \in \FW$ there is a biholomorphism of a neighborhood of $\overline{W}$ into some $\C^{n}$ that maps $W$ on a bounded, strongly pseudoconvex domain with smooth boundary; also $E$ is trivial on this neighborhood, and (b) the boundaries of $W \in \FW$ are in general position. This latter means that if $k \in \N$ and $\rho_{i}$ are smooth defining functions of $W_{i} \in \FW$, $i=1,\dots,k$, then $\llist d\rho. k.$ are linearly independent at each point of the set $\{\rho_{1}=\dots=\rho_{k}=0\}$.
\end{defn}
From Sard's theorem, it follows that there are arbitrarily fine spacial covers. \\

Suppose $\FV$ is a special cover of $X$, and $\FU$ is a countable refinement of $\FV$ consisting of Stein open sets. A refinement map $\gz : \FU \to \FV$ induces a linear map $Z^{qr}_{b}(\FV) \to Z^{qr}(\FU)$ by $g =(g_{V})_{V} \mapsto g | \FU = (g_{\gz(U)}|U)_{U}$. In \cite[Lemma 2.5]{Le3}, Lempert constructed homomorphisms (continuous linear maps),
\begin{align}
\ez = \ez^{qr} &: Z^{qr}(\FU) \to Z^{qr}_{b}(\FV) \text{, for } q, r \ge 0 \text{,}\nonumber \\
R=R^{qr} &: C^{qr}(\FU) \to C^{q-1,r}(\FU) \text{, for } q \ge 1 \text{ and } r \ge 0 \text{,}\nonumber \\
S=S^{qr} &: C^{qr}_{b}(\FV) \to C^{q,r-1}_{b}(\FV) \text{, for } q \ge 0 \text{ and } r \ge 1 \nonumber
\end{align}
with the following properties: \\
\begin{enumerate}[(a)]
\item if $f \in Z^{0r}(\FU)$, then $\ez f |\FU = f$;
\item if $q \ge 1$, then $\ez^{qr} = \dz S^{q-1,r+1} \ez^{q-1,r+1} \pz R^{qr}$;
\item if $q \ge 1$ and $f  \in Z^{qr}(\FU)$, then there exists $h \in C^{q-1,r}(\FU)$ such that $\pz h =0$ and $\ez f |\FU - f = \dz h$;
\item if $q \ge 1$, then $\dz R f + R \dz f = f$ on $C^{qr}(\FU)$;
\item if $r \ge 1$, then $\pz S f + S \pz f = f$ on $C^{qr}_{b}(\FV)$.
\end{enumerate}
In fact, Lempert claims (d) and (e) only for $f \in Z^{qr}(\FU)$ and $Z^{qr}_{b}(\FV)$ respectively. But from his definition of $R$ and $S$, (d) and (e) hold as stated here.

\begin{prop}
\label{prop2-1}
Let $\FV$ be a special cover of $X$ and $\FU$ a countable refinement of $\FV$ consisting of Stein open sets. Then the composition
\begin{align}
\label{Lhomo}
\begin{CD}
L : Z^{q0}(\FU) @> (\pz R)^{q} >> Z^{0q}(\FU) @> \cong >> Z^{q}_{\pz}(X,E)
\end{CD}
\end{align}
induces a topological isomorphism $\tilde{L} : H^{q}(\FU,E) \to H^{q}_{\pz}(X,E)$. Here $(\pz R)^{q} = \pz R^{1,q-1} \cdots \pz R^{q0}$, and the isomorphism in \eqref{Lhomo} is the inverse of the isomorphism $f \mapsto (f|U)_{U \in \FU}$.
\end{prop}

\begin{proof}
That $H^{q}(\FU,E)$ and $H^{q}_{\pz}(X,E)$ are isomorphic as vector spaces, of course follows from the local solvability of $\pz$ and from the softness of the sheaves $\Oz^{q}(X,E)$, i.e. from the existence of the operators $S$ and $R$. In fact, the vector space isomorphism obtained by diagram chasing will be our $\tilde{L}$. Since $L$ is continuous, so will be $\tilde{L}$, and we are left with showing $\tilde{L}^{-1}$ is also continuous. Consider the composition
\begin{align}
\label{Lzhomo}
\begin{CD}
\Lz : Z^{q}_{\pz}(X,E) \to Z^{0q}_{b}(\FV) @> (\dz S)^{q} >> Z^{q0}_{b}(\FV) \to Z^{q0}(\FU) \text{,}
\end{CD}
\end{align}
where the first map is $f \mapsto (f|V)_{V \in \FV}$ and $(\dz S)^{q} = \dz S^{q-1,1} \cdots \dz S^{0q}$; the last map is the refinement homomorphism induced by a refinement map $\FU \to \FV$, continuous by Cauchy estimates. Thus $\Lz$ itself is continuous.\\
\indent We compute $\Lz L$. Observe that the composition of the last map in \eqref{Lhomo} with the first map in \eqref{Lzhomo} is
\[\ez^{0q} : Z^{0q}(\FU) \to Z^{q}_{\pz}(X,E) \to Z^{0q}_{b}(\FV) \text{.} \]
Since $(\dz S)^{q}\ez^{0q}(\pz R)^{q} = \ez^{q0}$ by (b) above, we obtain $\Lz L$ as the composition
\[
\begin{CD}
\Lz L : Z^{q0}(\FU) @> \ez^{q0} >> Z^{q0}_{b}(\FV) \to Z^{q0}(\FU) \text{.}
\end{CD}
\]
In view of (c), therefore $\Lz L$ induces the identity on $H^{q}(\FU)$, and so $\tilde{L}^{-1}$ is continuous.
\end{proof}

\begin{proof}[Proof of Theorem \ref{thm2}]
By \cite[Corollary 2.6]{Le3}, if $\FW$ is a special cover of $X$ and $\FU$ is a Stein open cover of $X$ which is finer than $\FW$, then
\begin{align}
\label{lempert2}
H^{q}(\FU, E) \cong H^{q}_{b}(\FW, E) \cong H^{q}(X, E) \text{.}
\end{align}
By Proposition \ref{prop2-1} and \eqref{lempert2}, therefore $H^{q}_{\pz}(X,E) \cong H^{q}(X,E)$
as locally convex topological vector spaces.
\end{proof}


\section{Cohomology of Tensor Products}
\label{sec:3}
In this section, first we will recall the notion of tensor products. Then we consider a compact complex manifold $X$, a holomorphic Banach bundle $V \to X$ of finite rank, and relate the cohomology groups of $V$ to those of $V$ tensorized with a trivial Banach bundle. This we do by extending Hodge's decomposition theorem to a certain type of Banach bundles.\\

If $(A, \norm{\phantom{ab}}_{A})$ and $(B, \norm{\phantom{ab}}_{B})$ are Banach spaces, then there are many ways to give a topology to the tensor product $A \otimes B$. But if $\dim A = m < \wz$, there is no ambiguity, and in fact $A \otimes B \cong B \oplus \cdots \oplus B$, $m$ times. Concretely, if $A$ is a Banach space with a basis $\{a_{1}, \ldots, a_{m} \}$, then a norm on $A \otimes B$ can be given as
\[\norm{\sum_{i=1}^{m}a_{i}\otimes b_{i}}_{A\otimes B} = \max_{i} \norm{a_{i}}_{A} \norm{b_{i}}_{B} \text{, } b_{i} \in B \text{.}\]
If  $V \to X$ is a finite rank bundle and $E \to X$ is a Banach bundle, then the tensor product bundle $\pi : V \otimes E \to X$, whose fiber is $\pi^{-1}(x) = V_{x} \otimes E_{x}$, is well defined. If $T = X \times B \to X$ is a trivial Banach bundle, then we denote the tensor product bundle $V \otimes T \to X$ by $V \otimes B$. There is a canonical embedding of vector spaces
\begin{align}
\label{emb1}
\begin{split}
\frak{i} : \Oz^{q}(X,V)\otimes B &\to \Oz^{q}(X,V\otimes B) \\
\sum {f_{j} \otimes b_{j}} &\mapsto ((\xi_{1}, \ldots, \xi_{q}) \mapsto \sum {f_{j}(\xi_{1}, \ldots, \xi_{q}) \otimes b_{j}}) \text{,}
\end{split}
\end{align}
where $f_{j} \in \Oz^{q}(X,V)$, $b_{j} \in B$, and $\xi_{i} \in TX$.

\begin{lem}
\label{emb2}
$\frak{i}$ in \eqref{emb1} is injective and has dense range. For any finite dimensional subspace $F \st \Oz^{q}(X,V)$, the restriction $\frak{i}|F \otimes B$ is continuous and $\frak{i}(F \otimes B) \st \Oz^{q}(X,V \otimes B)$ is closed.
\end{lem}

\begin{proof}
Density follows from \cite[Proposition 2.1]{Li}, and continuity of $\frak{i}|F \otimes B$ is straightforward. As to injectivity, clearly it suffices to show that  $\frak{i}|F \otimes B$ is injective for any finite dimensional subspace $F \st \Oz^{q}(X,V)$. We shall verify this now, and at the same time show $\frak{i}(F \otimes B) \st \Oz^{q}(X,V \otimes B)$ is closed. Let $V^{*}$ be the dual bundle of $V$. Successively construct $f_{i} \in F$, $x_{i} \in X$, $\oz_{i} \in V^{*}_{x_{i}}$, and $\xi_{i1}, \ldots, \xi_{iq} \in T_{x_{i}}X$ for $i = 1, 2, \ldots, \text{dim} F$, so that
\[
\oz_{i} f_{j}(\xi_{i1}, \ldots, \xi_{iq})=
\left\{
\begin{array}{lr}
1 &\text{, if } i = j \phantom{.}\\
0 &\text{, if } i < j \text{.}
\end{array}
\right.
\]
Thus the $f_{i}$ form a basis of $F$. We claim that there are constants $c_{i}$ such that if
\begin{align}
\label{finiteclosed1}
f = \frak{i} \sum_{i=1}^{\text{dim} F}f_{i} \otimes b_{i} \text{, } b_{i} \in B \text{,}
\end{align}
then
\begin{align}
\label{finiteclosed2}
\norm{b_{i}}_{B} \le c_{i} \norm{f}_{0} \text{.}
\end{align}
This follows by induction on $i$, upon substituting $\xi_{i1}, \ldots, \xi_{iq}$ in \eqref{finiteclosed1} and applying $\oz_{i}$ for $i=1, 2, \ldots$. Thus we see $\frak{i}$ is injective. Next suppose that $b_{i}^{k} \in B$ are such that
\[ \lim_{k \to \infty} \frak{i} \sum_{i}f_{i} \otimes b_{i}^{k} = g \in \Oz^{q}(X,V \otimes B) \text{.}\]
Then \eqref{finiteclosed1}, \eqref{finiteclosed2} show that for each $i$ the $b_{i}^{k}$ form a Cauchy sequence. If $\lim_{k \to \infty}b_{i}^{k} = b_{i}$, then clearly $g = \frak{i} \sum f_{i} \otimes b_{i} \in \frak{i}(F \otimes B)$; that is, $\frak{i}(F \otimes B)$ is indeed closed.
\end{proof}
Henceforward we shall treat $\Oz^{q}(X,V)\otimes B$ as a subspace of $\Oz^{q}(X,V\otimes B)$. The main result of this section is the following.

\begin{thm}
\label{thm3}
Let $X$ be a compact complex manifold, $V \to X$ a holomorphic Banach bundle of finite rank, and $B$ a Banach space. Then the embedding $\frak{i}$ in \eqref{emb1} induces an isomorphism
\[
\begin{CD}
H^{q}(X, V) \otimes B @> \cong >> H^{q}(X,V \otimes B)
\end{CD}
\]
of locally convex topological vector spaces.
\end{thm}
Observe that $\dim H^{q}(X,V) < \wz$, so the topology on $H^{q}(X,V) \otimes B$ is defined as above. The special case $H^{q}(X,V)=0$ was earlier proved by Leiterer in \cite{Li}. Theorem \ref{thm3} shows that for bundles of infinite rank, we cannot expect a finiteness theorem as \eqref{hodge1}. \\

We start by introducing notation from Hodge theory. Let $X$ be a compact manifold and $V \to X$ a finite rank vector bundle. Fix hermitian metrics on $X$ and $V$; they determine an inner product $(f,g)$ on $\Oz^{q}(X,V)$. With the adjoint $\pz^{*}$ of $\pz$, we set
\[\bx = \bx_{q} = \pz_{q-1}\pz^{*}_{q} + \pz^{*}_{q+1}\pz_{q} : \Oz^{q}(X,V) \to \Oz^{q}(X,V) \text{.}\]
By Hodge's theorem (see e.g. \cite{W}), there are the orthogonal projection $H = H_{q} : \Oz^{q}(X,V) \to H \Oz^{q}(X,V)$ to harmonic $(0,q)$-forms, and Green's operator $G = G_{q} : \Oz^{q}(X,V) \to \Oz^{q}(X,V)$ satisfying
\begin{align}
\label{hodged}
\id = H + \bx G \text{, } H \pz = \pz H = 0 \text{, and } G \pz = \pz G \text{.}
\end{align}

It is not known whether for general Banach bundles there is anything like a Hodge decomposition. But we will use \eqref{hodged} to define such a decomposition for bundles of form $V \otimes B$.

\begin{lem}
\label{exte}
Let $X$ be a complex manifold, $V \to X$ a holomorphic Banach bundle of finite rank, and $B$ a Banach space. Given any $q$ and $q'$, if $T : \Oz^{q}(X,V) \to \Oz^{q'}(X,V)$ is a continuous linear operator, then there exists a continuous linear operator
\[T^{B} : \Oz^{q}(X,V \otimes B) \to \Oz^{q'}(X,V \otimes B)\]
such that $T^{B} = T \otimes \id_{B}$ on $\Oz^{q}(X,V) \otimes B$.
\end{lem}

\begin{proof}
In \cite[Proposition 2.2]{Li}, Leiterer proved this for operators acting between $C^{k}$ and $C^{l}$ forms. The same proof gives the result for smooth forms. Alternatively, Leiterer's Proposition 2.2 implies our Lemma, since the continuity of $T$ means that for every $l \in \N$ there is a $k \in \N$ so that $T$ extends to a continuous linear operator $C^{k}_{q}(X,V) \to C^{l}_{q'}(X,V)$.
\end{proof}

Denote $L = L_{q} = \pz^{*} G_{q} : \Oz^{q}(X,V) \to \Oz^{q-1}(X,V)$. We shall apply Lemma \ref{exte} to the operators $H$ and $L$. Note that \eqref{hodged} implies that on  $\Oz^{q}(X,V)\otimes B$
\begin{align}
\label{hodge2}
\begin{split}
\id &= H \otimes \id_{B} + \bx G \otimes \id_{B} \\
&= H \otimes \id_{B} +  \pz_{V} L_{q} \otimes \id_{B} + L_{q+1} \pz_{V} \otimes \id_{B} \text{.}
\end{split}
\end{align}

\begin{lem}
Let $X$ be a compact complex manifold, $V \to X$ a holomorphic Banach bundle of finite rank, and $B$ a Banach space. Then the terms on the right hand side in \eqref{hodge2} can be extended to continuous linear operators on $\Oz^{q}(X,V \otimes B)$, and the extended operators will satisfy the same identity.
\end{lem}

\begin{proof}
$L$ and $H$ are continuous on $\Oz^{q}(X,V)$ (see e.g. \cite[Chapter 4, Section 4]{W}). By Lemma \ref{exte} there exist continuous linear operators  $L^{B}$ and $H^{B}$ on $\Oz^{q}(X,V \otimes B)$ such that  $L^{B} = L^{B}_{q} = L_{q} \otimes \id_{B}$ and $H^{B} = H \otimes \id_{B}$ on $\Oz^{q}(X,V) \otimes B$. As $\Oz^{q}(X,V) \otimes B$ is dense in $\Oz^{q}(X,V \otimes B)$, the identity,
\begin{align}
\label{hodge3}
\id = H^{B} +  \pz_{V\otimes B} L^{B}_{q} + L^{B}_{q+1} \pz_{V\otimes B} \text{,}
\end{align}
follows.
\end{proof}

\begin{prop}
\label{iso1}
Let $X$ be a compact complex manifold, $V \to X$ a holomorphic Banach bundle of finite rank, and $B$ a Banach space. Then
\begin{align}
\label{HB}
H^{B} : Z^{q}_{\pz}(X, V \otimes B) \to H^{B} \Oz^{q}(X,V\otimes B)
\end{align}
descends to an isomorphism
\[H^{q}_{\pz}(X, V \otimes B) \cong H^{B} \Oz^{q}(X,V\otimes B) \st \Oz^{q}(X, V \otimes B)\]
of locally convex topological vector spaces.
\end{prop}

\begin{proof}
\eqref{hodged} and density imply $\pz H^{B} = H^{B} \pz = 0$, and $H^{B} = H^{B}H^{B}$. This shows first that $H^{B} \Oz^{q}(X,V\otimes B) = H^{B} H^{B} \Oz^{q}(X,V\otimes B) \st H^{B} Z^{q}_{\pz}(X,V\otimes B)$, so \eqref{HB} is a continuous surjection; and second that \eqref{HB} descends to a continuous surjection $H^{q}_{\pz}(X, V \otimes B) \to H^{B} \Oz^{q}(X,V\otimes B)$. This latter is injective by \eqref{hodge3}. It also has a continuous inverse, because with the seminorms $\norm{\phantom{ab}}_{K,k}$ in \eqref{norm1} and the corresponding quotient seminorms on $H^{q}_{\pz}(X, V \otimes B)$ one can estimate the size of the cohomology class $[f]$ of any $f \in Z^{q}_{\pz}(X,V \otimes B)$ as
\[
\norm{[f]}_{K,k} = \inf_{g}{\norm{f + \pz g}}_{K,k} \le \norm{f - \pz L^{B} f}_{K,k} = \norm{H^{B}f}_{K,k} \text{.}
\]
\end{proof}

Now we are ready to prove Theorem \ref{thm3}.

\begin{proof}[Proof of Theorem \ref{thm3}]
In view of Proposition \ref{iso1}, it suffices to show that by restricting $\frak{i}$ of \eqref{emb1}, we obtain a topological isomorphism
\[
\tilde{\frak{i}} : (H \Oz^{q}(X,V)) \otimes B \to H^{B} \Oz^{q}(X,V \otimes B) \text{.}
\]
Lemma \ref{emb2} implies that $\tilde{\frak{i}}$ is continuous. Since $\frak{i}$ is injective, so is $\tilde{\frak{i}}$. By Lemma \ref{emb2} the range of $\frak{i}$ is dense in $\Oz^{q}(X,V \otimes B)$; hence the range of $H^{B}\frak{i}$, i.e. the range of $\tilde{\frak{i}}$, is dense in $H^{B} \Oz^{q}(X,V \otimes B)$. By Lemma \ref{emb2} the range is also closed, so it must be $H^{B}\Oz^{q}(X,V \otimes B)$, and by the open mapping theorem $\tilde{\frak{i}}$ is a topological isomorphism.
\end{proof}


\section{The Splitting of Banach Bundles}
\label{sec:4}
We observed that cohomology groups of Banach bundles can be infinite dimensional. Still, Leiterer and Lempert have proved finiteness theorems for a certain type of Banach bundles. Lempert's finiteness theorem will be the starting point in the proof of our splitting theorem.

\begin{thmq}[Lempert \cite{Le3}]
Let $X$ be a compact complex manifold, $E$, $F$ holomorphic Banach bundles over $X$ that are compact perturbations of one another, and $q = 0, 1, \ldots$. If $H^{q+1}(X,E)$ is Hausdorff and $\dim H^{q}(X,E) < \infty$, then $H^{q}(X,F)$ is also finite dimensional (and Hausdorff).
\end{thmq}

\begin{cor}
\label{finite}
Let $X$ be a compact complex manifold, $V \to X$ a holomorpihc Banach bundle of finite rank, and  $E \to X$ a holomorphic Banach bundle that is a compact perturbation of a trivial Banach bundle $T \to X$. If $H^{q}(X,V) = 0$, then $\dim H^{q}(X,V \otimes E) < \infty$.
\end{cor}
\begin{proof}
$V\otimes E$ is a compact perturbation of $V\otimes T$. By Theorem \ref{thm3}, $H^{q+1}(X,V\otimes T)$ is Hausdorff and $H^{q}(X,V\otimes T)$ is finite dimensional. Hence the Corollary follows from Lempert's Theorem.
\end{proof}
In \cite{Li}, Leiterer proved the above Corollary, assuming a mild condition (the compact approximation property) on the fibers of $E$. \\

Our Splitting Theorem \ref{thm1} is a consequence of the following two propositions.
\begin{prop}
\label{mprop1}
Let $X$ be a compact complex manifold and $E \to X$ a holomorphic Banach bundle that is a compact perturbation of a trivial bundle. If $H^{1}(X,\O)=0$ then $E$ has a trivial subbundle of finite corank.
\end{prop}

\begin{prop}
\label{mprop2}
Let $X$ be a compact complex manifold, $E \to X$ a holomorphic Banach bundle, and $T \st E$ a trivial subbundle of finite corank. Then $E$ has a subbundle $F$ of finite rank and $T$ has a trivial subbundle $T'$ of finite corank such that $E = T' \oplus F$.
\end{prop}

We start with the proof of Proposition \ref{mprop2}.

\begin{proof}[Proof of Proposition \ref{mprop2}]
Let $T$ be the bundle $X \times B \to X$ with $B$ a Banach space, and write $G = E / T$. With $A \st B$ a closed subspace to be specified and $S \st T$ the subbundle $X \times A \to X$, the quotient map $E \to E/S$ induces a homomorphism of short exact sequences of Banach bundles
\[
\begin{CD}
0 @>>> T @>>> E @>>> G @>>> 0 \phantom{,}\\
@. @VVV @VVV @VV \id V @. \\
0 @>>> T/S @>>> E/S @>>> G @>>> 0 \text{,}
\end{CD}
\]
and also of
\begin{align}
\label{soc0}
\begin{CD}
0 @>>> \Hom(G,T) @>>> \Hom(G,E) @>>> \Hom(G,G) @>>> 0 \phantom{.} \\
@.               @VVV           @VVV           @VV \id V   @.\\
0 @>>> \Hom(G,T/S) @>>> \Hom(G,E/S) @>>> \Hom(G,G) @>>> 0 \text{.}
\end{CD}
\end{align}
Let $G^{*}$ be the dual bundle of $G$. Since $\Hom(G,T) \cong G^{*} \otimes B$ and $\Hom(G,T/S) \cong G^{*} \otimes B/A$, Theorem \ref{thm3} gives us a diagram
\begin{align}
\label{soc2}
\begin{CD}
H^{1}(X,\Hom(G,T)) @> \cong >> H^{1}(X,G^{*}) \otimes B \phantom{.} \\
@V \pi VV                       @VVV \\
H^{1}(X,\Hom(G,T/S)) @> \cong >> H^{1}(X,G^{*}) \otimes B/A \text{.}
\end{CD}
\end{align}
With the right vertical arrow induced by the quotient map $B \to B/A$, the diagram \eqref{soc2} is commutative since the embedding $\frak{i}$ in \eqref{emb1} is functorial in $B$. From \eqref{soc0}, portions of the associated exact sequences in cohomology give a commutative diagram
\begin{align}
\label{soc1}
\begin{CD}
@. \Gz(X,\Hom(G,G)) @> \az >> H^{1}(X,\Hom(G,T)) \phantom{.} \\
@.                  @V \id VV   @VV \pi V \\
\Gz(X,\Hom(G,E/S)) @> \bz >> \Gz(X,\Hom(G,G)) @> \gz >> H^{1}(X,\Hom(G,T/S)) \text{.}
\end{CD}
\end{align}
Consider the section $h$ of $\Hom(G,G)$ corresponding to the identity homomorphism $G \to G$. Its image $\az(h)$ in $H^{1}(X,\Hom(G,T)) \cong H^{1}(X,G^{*}) \otimes B$ can be written $\az(h) = \sum_{1}^{m}g_{i} \otimes b_{i}$, $g_{i} \in H^{1}(X,G^{*})$, $b_{i} \in B$. If we choose $A \st B$ to be the span of the $b_{i}$, then \eqref{soc1} shows that $0 = \pi \left( \az (h) \right) = \gz (h)$, and so $h$ is in the range of $\bz$. In other words, the identity homomorphism $G \to G$ is covered by a homomorphism $G \to E/S$; the image $G' \st E/S$ of this latter has then finite rank and is complementary to $T/S \st E/S$. If we now choose $F \st E$ to be the preimage of $G'$ under the quotient map $E \to E/S$, and $T' \st T$ to be complementary to $S \st T$, then $E = T' \oplus F$, as claimed.
\end{proof}

The following lemma is not new; for lack of reference we include a proof.

\begin{lem}
\label{bu}
Let $\mu : L \to \P_{n}$ be a negative line bundle, $h : L \to \R$ a negatively curved hermitian metric, and $D=\{ h < 1 \}$. Then $H^{q}(D, \O) = 0$ for $q > 0$.
\end{lem}

\begin{proof}
Let $S^{1} = \S$, then $S^{1}$ acts continuously on $L$ by
\[\rho_{t}(v) = e^{2 \pi i t} v \text{ , } t \in S^{1} \text{ , } v \in L \text{.} \]
Let $\FU =\{U_{i} \}$ be a Stein cover of $\P_{n}$, then $\FV =\{V_{i}=D \cap \mu^{-1}U_{i} \}$ is a Stein cover of $D$. For cocycle $f \in Z^{q}(\FV, \O)$, set
\[f_{k} = \int_{S^{1}} e^{- 2 \pi i k t} \rho_{t}^{*}f dt \text{, for } k \ge 0 \text{,} \]
so that $f = \sum_{k=0}^{\infty} f_{k}$ is the fiberwise homogeneous expansion of $f$. Now the $k$-homogeneous cocycle $f_{k}$ can be thought of as a cocyle in $Z^{q}(\FU, L^{-k})$, and therefore it is a coboundary (for example by Kodaira's vanishing theorem). Therefore $[f] \in H^{q}(\FV, \O)$ is the limit of the classes $\sum_{k=0}^{m}[f_{k}] = 0 \in H^{q}(\FV, \O)$. On the other hand, the condition on $h$ means that $D$ has strictly pseudoconvex boundary, hence by Grauert's theorem (see \cite{Ga}), $H^{q}(D,\O) \cong H^{q}(\FV, \O)$ is finite dimensional and so Hausdorff. Therefore $[f]=0$ as claimed.
\end{proof}

\begin{lem}
\label{bux}
Let $E \to X$ be a holomorphic Banach bundle over a compact complex manifold, $\pi : \tilde{X} \to X$ the blow up at $p \in X$, and $\tilde{E}=\pi^{*}E \to \tilde{X}$ the pull back bundle. If $\I_{p}^{E}$ and $\I_{S}^{\tilde{E}}$ are the ideal sheaves of germs of $E$- (respectively $\tilde{E}$-) valued holomorphic functions vanishing at $p$ (respectively $S=\pi^{-1}(p)$), then
\[\pi^{*} : H^{q}(X, E) \to H^{q}(\tilde{X}, \tilde{E}) \text{ and } \pi^{*}_{p} : H^{q}(X, \I_{p}^{E}) \to H^{q}(\tilde{X}, \I_{S}^{\tilde{E}})\]
are isomorphisms for $q = 0$ and monomorphisms for $q = 1$; if $E$ is of finite rank, then they are isomorphisms for any $q \ge 0$.
\end{lem}

\begin{proof}
Let $\dim X = n \ge 2$ and $\rank E = r \le \infty$. (If $\dim X = 1$, the Lemma is obvious.) The point $p$ has a arbitrarily small neighborhood $V \st X$, biholomorphic to a ball, such that $U = \pi^{-1}V \st \tilde{X}$ is biholomorphic to the disc bundle $\{h < 1\}$ of a hermitian line bundle $(L,h)$ over $S \cong \P_{n-1}$. Choose $V$ so that, in addition, $E|V$ is trivial. Then $\tilde{E}|U$ is also trivial. Consider the following commutative diagram
\begin{align}
\label{cdiagram}
\begin{CD}
\Gz(V, \I_{p}^{E})         @>i>> \Gz(V, E) \phantom{,}\\
@V\pi^{*}_{p}VV                   @VV\pi^{*}V \\
\Gz(U, \I_{S}^{\tilde{E}}) @>\tilde{i}>> \Gz(U, \tilde{E}) \text{,}
\end{CD}
\end{align}
where $i$ and $\tilde{i}$ are inclusion maps and $\pi^{*}_{p}$ is the restriction of $\pi^{*}$. Both $\pi^{*}$ and $\pi^{*}_{p}$ are isomorphisms. The inverse of $\pi^{*}$ is obtained by first associating with $f \in \Gz(U,\tilde{E})$ the section
\[ \left(\pi^{-1}|V \backslash \{p\}\right)^{*} f \in \Gz(V \backslash \{p\}, E)\]
and then extending this section, by Hartogs' theorem, to $p$. The restriction of this inverse to $\Gz(U, \I_{S}^{\tilde{E}})$ is then the inverse of $\pi^{*}_{p}$. Next let $\FW$ be a Stein cover of $X$ so that $V \in \FW$ and no $W \in \FW \backslash \{V\}$ contains $p$. Let $\FU = \{ \pi^{-1}W | W \in \FW \}$. That $\pi^{*}$ and $\pi^{*}_{p}$ in \eqref{cdiagram} are isomorphisms implies that
\begin{align}
\label{buiso1}
H^{q}(\FW,E) \cong H^{q}(\FU, \tilde{E}) \text{ and } H^{q}(\FW,\I_{p}^{E}) \cong H^{q}(\FU, \I_{S}^{\tilde{E}}) \text{.}
\end{align}
Because $\FW$ is Stein, by Leray's theorem,
\begin{align}
\label{buiso2}
H^{q}(\FW,E) \cong H^{q}(X, E) \text{ and } H^{q}(\FW,\I_{p}^{E}) \cong H^{q}(E, \I_{p}^{E}) \text{.}
\end{align}
Since the canonical maps
\begin{align}
\label{buinj}
H^{q}(\FU,\tilde{E}) \to H^{q}(\tilde{X},\tilde{E}) \text{ and }  H^{q}(\FU, \I_{S}^{\tilde{E}}) \to H^{q}(\tilde{X}, \I_{S}^{\tilde{E}})
\end{align}
are isomorphisms for $q = 0$ and monomorphisms for $q = 1$, by combining \eqref{buiso1}, \eqref{buiso2}, and \eqref{buinj},
\[\pi^{*} : H^{q}(X, E) \to H^{q}(\tilde{X}, \tilde{E}) \text{ and } \pi^{*}_{p} : H^{q}(X, \I_{p}^{E}) \to H^{q}(\tilde{X}, \I_{S}^{\tilde{E}})\]
are isomorphisms for $q = 0$ and monomorphisms for $q = 1$. \\
\indent Now let us assume $\rank E = r < \infty$. For $j \ge 1$
\begin{align}
\label{nonstein}
H^{j}(V,E)=H^{j}(V,\I_{p}^{E})=H^{j}(U,\tilde{E})=H^{j}(U,\I_{S}^{\tilde{E}})=0 \text{.}
\end{align}
The first two groups vanish simply because $V$ is Stein, and the third vanishes by Lemma \ref{bu}, since $\tilde{E}|U$ is trivial. The last group fits in the exact sequence
\begin{align}
\label{Svanish}
\begin{CD}
H^{j-1}(U,\tilde{E})@>\az>>H^{j-1}(U,\tilde{E}/\I_{S}^{\tilde{E}})@>\bz>>H^{j}(U,\I_{S}^{\tilde{E}}) @>>>H^{j}(U,\tilde{E}) = 0 \text{.}
\end{CD}
\end{align}
Since $\tilde{E}|U$ is trivial,
\[H^{j-1}(U,\tilde{E}/\I_{S}^{\tilde{E}}) \cong H^{j-1}(S,\O^{\oplus r}) \cong H^{j-1}(S,\O)\otimes \C^{r} \text{.}\]
Thus $H^{0}(U,\tilde{E}/\I_{S}^{\tilde{E}})$ is identified with the space of constant $\C^{r}$-valued functions over the compact $S$, and $\az$ is clearly surjective for $j=1$. For $j>1$, $H^{j-1}(S,\O)=0$ by Kodaira's vanishing theorem. Therefore $\bz=0$ and $H^{j}(U,\I_{S}^{\tilde{E}})=0$ by \eqref{Svanish}. Although $\FU$ is not a Stein cover, Leray's theorem still implies
\begin{align}
\label{buiso3}
H^{q}(\FU,\tilde{E}) \cong H^{q}(X, \tilde{E}) \text{ and } H^{q}(\FU,\I_{S}^{\tilde{E}}) \cong H^{q}(E, \I_{S}^{\tilde{E}}) \text{,}
\end{align}
because the only non-Stein intersection of elements of $\FU$ is $U$ itself, which is acyclic according to \eqref{nonstein}. Putting together \eqref{buiso1}, \eqref{buiso2}, and \eqref{buiso3},
\[H^{q}(X, E) \cong H^{q}(\tilde{X}, \tilde{E}) \text{ and } H^{q}(X, \I_{p}^{E}) \cong H^{q}(\tilde{X}, \I_{S}^{\tilde{E}}) \text{.}\]
\end{proof}

\begin{lem}
\label{dimp}
Let $X$ be a compact complex manifold with $H^{1}(X, \O) = 0$, and $E \to X$ a compact perturbation of a trivial Banach bundle. If $\I_{p}^{E}$ is the ideal sheaf of germs of $E$-valued holomorphic functions vanishing at $p \in X$, then for $q=0,1$
\[\dim H^{q}(X, \I_{p}^{E}) < \infty \text{.}\]
\end{lem}

\begin{proof}
Let $\pi : \tilde{X} \to X$ be the blow up at $p \in X$ and $S=\pi^{-1}(p)$, $\tilde{E}$, and $\I_{S}^{\tilde{E}}$ as in Lemma \ref{bux}. Further, let $\I_{p}$ be the ideal sheaf of $p \in X$ and $\I_{S}$ the ideal sheaf of $S \st \tilde{X}$. $\I_{S}$ is isomorphic to the sheaf of sections of a line bundle $L \to \tilde{X}$. By Lemma \ref{bux} $H^{q}(\tilde{X},L) \cong H^{q}(\tilde{X}, \I_{S}) \cong H^{q}(X,\I_{p})$. This latter group vanishes for $q=0$ by the maximum principle, but also for $q=1$, as follows from the exact sequence
\[
\begin{CD}
H^{0}(X,\O) @> \az >> H^{0}(X,\O/\I_{p}) @>>> H^{1}(X,\I_{p}) @>>> H^{1}(X,\O) = 0 \text{,}
\end{CD}
\]
where $\az$ is surjective. Since $\tilde{E}$ is a compact perturbation of a trivial Banach bundle, Corollary \ref{finite} implies $H^{q}(\tilde{X},\I_{S}^{\tilde{E}}) \cong H^{q}(\tilde{X},L \otimes \tilde{E})$ is finite dimensional for $q = 0, 1$, and by Lemma \ref{bux} so is $H^{q}(X, \I_{p}^{E})$.
\end{proof}

We prove Proposition \ref{mprop1} through these Lemmas. With the Banach space $\Gz(X,E)$ and given $p \in X$, define homomorphisms
\begin{align}
&\ez : X \times \Gz(X,E) \ni (x,s) \mapsto s(x) \in E \text{,} \nonumber \\
&\ez_{p} : \Gz(X,E) \ni s \mapsto s(p) \in E_{p} \text{.} \nonumber
\end{align}

\begin{lem}
\label{mlem1}
If $X$ is a compact complex manifold with $H^{1}(X, \O)=0$ and $E \to X$ a compact perturbation of a trivial bundle, then $\ez_{p}$ is Fredholm for every $p \in X$.
\end{lem}

\begin{proof}
If in the exact sequence
\[
\begin{CD}
H^{0}(X,\I_{p}^{E}) @>>> H^{0}(X,E) @>\az>> H^{0}(X,E/\I_{p}^{E}) @>>>H^{1}(X,\I_{p}^{E}) \text{,}
\end{CD}
\]
the space $H^{0}(X,E/\I_{p}^{E})$ is identified with $E_{p}$ (and $H^{0}(X,E)$ with $\Gz(X,E)$), then $\az$ becomes $\ez_{p}$. Since the spaces at the two extremes are finite dimensional by Lemma \ref{dimp}, $\ez_{p}$ is indeed Fredholm.
\end{proof}

\begin{lem}
\label{mlem2}
Let $X$ be a compact complex manifold, and $E \to X$ a holomorphic Banach bundle.
\begin{enumerate}[(a)]
\item If $\ez_{p}$ is Fredholm for some $p \in X$, then there are a neighborhood $U \st X$ of $p$ and a finite codimensional subspace $A \st \Gz(X,E)$ such that $\ez_{x}|A$ is injective and $\ez_{x} A \st E_{x}$ is closed for $x \in U$. \\
\item Let $A \st \Gz(X,E)$ be a closed subspace, and $U \st X$ open. If $\ez_{x}|A$ is injective and $\ez_{x} A \st E_{x}$ is closed and finite codimensional for $x \in U$, then $\ez$ defines an isomorphism between the bundle $U \times A \to U$ and a (necessarily trivial) subbundle $T \st E|U$ of finite corank.
\end{enumerate}
\end{lem}

\begin{proof}
(a) Let $A$ be a complementary subspace to $\Ker \ez_{p}$. Let $\fz : E|_{V} \to V \times B$ be a trivialization, $p \in V \st X$. Denote by $\tilde{\fz} : E|_{V} \to B$ the composition of $\fz$ with the projection $ V \times B \to B$. Given $x \in V$, let $\tilde{\ez}_{x}$ be the homomorphism
\begin{align}
\label{ep}
\tilde{\ez}_{x} : A \ni s \mapsto \tilde{\fz}s(x) \in B \text{.} \nonumber
\end{align}
By the open mapping theorem, there is a constant $c > 0$ so that $\norm{s}_{A} \le c \norm{\tilde{\ez}_{p}s}_{B}$ for any $s \in A$. Then
\[ \norm{s}_{A} \le c \norm{\tilde{\ez}_{x}s + (\tilde{\ez}_{p} - \tilde{\ez}_{x})s}_{B} \le c \norm{\tilde{\ez_{x}}s}_{B} + \frac{\norm{s}_{A}}{2} \text{,}\]
if $x$ is sufficiently close to $p$, hence the claim. \\

(b) We will prove this without the assumption that $X$ is compact. The advantage is then that, the statement being local, we can assume $X = U \st \C^{n}$ is open and $E = X \times B \to X$ is trivial. We will think of elements of $\Gz(X,E)$ as $B$-valued holomorphic functions on $X$. Our hypothesis still implies that $A$ is a Banach subspace of $\Gz(X,E)$, in fact one isomorphic to a subspace of $B$. Choose $h_{1}, \ldots, h_{m} \in \Gz(X,E)$ so that for a fixed $p \in X$, the $h_{i}(p)$ span a complementary subspace to $\ez_{p}A \st B$. Define a homomorphism $\phi$ of the trivial bundles $F = X \times (A \oplus \C^{m}) \to X$ and $E \to X$ by
\[\phi (x, s, \bz) = \left(x,s(x) + \sum_{i=1}^{m}\bz_{i}h_{i}(x)\right) \in X \times B \text{,}\]
where $x \in X$, $s \in A$, and $\bz=(\bz_{1}, \ldots, \bz_{m}) \in \C^{m}$. The differential of $\phi$ at $(p,0)$,
\[d \phi_{(p,0)} (\xi, s, \bz) = \left(\xi , s(p) + \sum_{i=1}^{m}\bz_{i}h_{i}(p)\right) \]
is an isomorphism between $\C^{n} \times (A \oplus \C^{m})$ and $\C^{n} \times B$ by the choice of $h_{i}$. By the implicit function theorem therefore $\phi$ is invertible on a neighborhood $V \times N \st X \times (A \oplus \C^{r})$ of $(p,0)$. In fact, the inverse can be extended to all of $V \times B$ by linearity. The upshot is that $\phi$ is an isomorphism of the bundles $F|V$ and $E|V$. Since the restriction of $\phi$ and $\ez$ to $V \times (A \oplus (0))$ agree, this implies the claim.
\end{proof}

\begin{proof}[Proof of Proposition \ref{mprop1}]
By Lemma \ref{mlem1} and Lemma \ref{mlem2} (a), for each $p \in X$, there are a neighborhood $U_{p} \st X$ of $p$ and a finite codimensional subspace $A_{p} \st \Gz(X,E)$ such that $\ez_{x}|A_{p}$ is injective for $x \in U_{p}$. Lemma \ref{mlem1} also implies $\codim \ez_{x} A_{p} < \infty$. Let $\{ U_{p} \text{ } | \text{ } p \in P\}$ be a finite cover of $X$, and  $A = \bigcap_{p \in P}A_{p}$. Then $\ez_{x} | A$ is injective for all $x \in X$. Moreover $\ez_{x} A \st E_{x}$ is closed and finite codimensional. By Lemma \ref{mlem2} (b), $T= \ez A \st E$ is a trivial subbundle of finite corank, which proves the Proposition.
\end{proof}


\end{document}